\documentclass{amsart}

\usepackage{amssymb, amsmath, latexsym, a4}
\usepackage[latin1]{inputenc}
\usepackage[all]{xy}
\usepackage{anysize}

\marginsize{3.5 cm}{3.5 cm}{2.5 cm}{2.5 cm}

\newtheorem{De}{Definition}
\newtheorem{Th}[De]{Theorem}

\newtheorem{Ex}[De]{Example}

\let\To\Rightarrow

\def\xto#1{\xrightarrow[]{#1}}
\def\ta{{\mathcal T}}
\def\ka{{\mathcal K}}
\def\ca{{\mathcal C}}

\newcommand\bI{{\boldsymbol I}}
\newcommand\F{{\boldsymbol F}}

\def\id{\sf Id}
\let\toto\rightrightarrows

\begin{document}

\title{Schreier theory of track categories}

\author[M. Pirashvili]{Mariam  Pirashvili}

\maketitle

\section{Introduction}

This paper is a continuation of our study of non-abelian Baues-Wirsching cohomologies. In our previous paper \cite{mp}, we defined second non-abelian cohomology $H^2(\ca, D)$ of a small category $\ca$ with coefficients in a so-called centralised natural system $D$. We proved that $H^2(\ca, D)$ classifies linear extensions of $\ca$ by $ D$, generalising the corresponding result for abelian natural systems proved in \cite{bw}.

For an abelian natural system $D$, the third cohomology classifies certain abelian track categories \cite{bj}. A track category is a $2$-category where all $2$-morphisms are isomorphisms. A track category is called abelian if for every $1$-morphism $f$, the group $Aut(f)$ is abelian.

In a similar fashion to the above, we want to generalise this result for non-abelian track categories. In this paper we solve this problem for the folowing important case: Given categories $\ka$ and $\ca$ and a functor $\pi: \ka\to\ca$, which is identity on objects and surjective on morphisms, and $G$, a centralised natural system of groups on $\ka$, we describe the equivalence classes of all track categories $\ta$ for which $\ka$ is the underlying category and $\ca$ is the homotopy category and $G_f=Aut(f)$.

\section{Preliminaries}
\subsection{Track categories}
For a small groupoid $\bf G$, the set of its objects will be denoted by ${\bf G}_0$ and the set of morphisms by ${\bf G}_1$. We have the source and target maps $
\xymatrix{
{\bf G}_1\ar@<1ex>[r]^s\ar@<-1ex>[r]_t&\ {\bf G}_0
}
$. A groupoid is called \emph{abelian} if all automorphism groups are abelian.

\begin{De}
A category enriched in groupoids $\ta$, also termed \emph{track category}, is a $2$-category whose all $2$-cells are invertible. It is thus a class of objects $Ob(\ta)$, a collection of groupoids $\ta(A,B)$ for $A,B\in Ob(\ta)$ called \emph{hom-groupoids of $\ta$}, identities $1_A\in \ta(A,A)_0$ and composition functors $\ta(B,C)\times \ta(A,B)\to \ta(A,C)$ satisfying the usual equations of associativity and identity morphisms.
\end{De}

Thus  $\ta(A,B)_0$  denotes the set of objects of the groupoid $\ta(A,B)$, while  $\ta(A,B)_0$ denotes the set morphisms of the groupoid $\ta(A,B)$. Elements of  $\ta(A,B)_0$ we will call morphisms of  $\ta$, while elements of  $\ta(A,B)_1$ are called tracks of $\ta$. The underlying category of $\ta$ is denoted by $\ka$ 

For $f,g\in\ta(A,B)_1$ we shall write $f\simeq g$ (and say $f$ is \emph{homotopic} to $g$) if there exists a morphism $\alpha$ of $\ta$ from  $f$ to $g$. Occassionally this will also be denoted by $\alpha :f\simeq g$ or $\alpha :f\To g$. Sometimes $\alpha$ is called a \emph{homotopy} or a \emph{track} from $f$ to $g$. Homotopy is a natural equivalence relation on morphisms of $\ka$ and determines the homotopy category $\ca=\ka/\simeq$. Objects of $\ca$ are once again the objects in $Ob(\ta)$, while the morphisms of $\ca$ are the homotopy classes of morphisms in $\ka$. We denote the canonical functor by $\pi$. So $\pi:\ka\to\ca$.

\

A morphism $g:B\to C$ in $\ta$  induces the functors
$$g_*:\ta (A,B)\to \ta (A,C), \ \ f\mapsto gf, \ \ \alpha\mapsto g_*\alpha,$$
$$g^*: \ta(C,D)\to \ta(B,D), \ \ h\mapsto hg, \ \ \beta\mapsto g^*\beta.$$
These functors are restrictions of the composition functors. It follows from the definition that the following relations hold:
$$(\alpha +\beta)+\gamma=\alpha +(\beta+\gamma),\leqno {\rm TR \ 1}$$
$$\alpha +0=\alpha=0+\alpha,   \leqno {\rm TR \ 2}$$
$$f^*(\alpha +\beta)=f^*(\alpha) +f^*(\beta),\leqno {\rm TR \ 3} $$
$$g_*(\alpha +\beta)=g_*(\alpha) +g_*(\beta),\leqno {\rm TR \ 4}$$
$$f^*(0)=0=g_*(0),\leqno {\rm TR \ 5}$$
$$(ff_1)^*=f_1^*f^*, \ 1^*=1,\leqno {\rm TR \ 6}$$
$$(gg_1)_*=g_*g_{1*}, \ 1_*=1,\leqno {\rm TR \ 7}$$
$$g_*f^*=f^*g_*,\leqno {\rm TR \ 8}$$
$$g_*(\alpha )+f_1^*(\alpha _1)=f^*(\alpha _1)+ g_{1*}(\alpha). \leqno {\rm TR \ 9}$$
The following diagram explains the 1-morphisms and $2$-morphisms in TR 9:
$$\xymatrix{A\ar@/^/[r]^{f} \ar@/_/[r]_{f_1}&B\ar@/^/[r]^{g}\ar@/_/[r]_{g_1} & C}, \ \ \alpha:f\To f_1, \ \ \alpha_1:g \To g_1. $$
The equality TR 9 holds in $\ta(gf,g_1f_1)$.

\begin{De}
A \emph{track functor}, or else $2$-functor, $F:\ta\to \ta'$ between track categories is a groupoid enriched functor. Thus $F$ assigns to each $A\in Ob(\ta)$ an object $F(A)\in Ob(\ta')$, to each map $f:A\to B$ in $\ta$ a map $F(t):F(A)\to F(B)$, and to each track $\alpha:f\To g$ for $f,g:A\to B$ a track $F(\alpha):F(f)\To F(g)$ in a functorial way, i.e. so that one gets functors $F_{A,B}:\ta(A,B)\to \ta'(F(A),F(B))$. These assignments are compatible with identities and composition, or equivalently induce a functor $\ta_1\to \ta'_1$, that is $F(1_A)=1_{F(A)}$ for $A\in Ob(\ta)$, $F(fg)=F(f)F(g)$, and $F(\phi\times\psi)=F(\phi)\times(F\psi)$ for any $\phi: f\To f'$, $\psi: g\To g',$, $f,f':B\to C$, $g,g':A\to B$ in $\ta$.
\end{De}

A track functor $F:\ta \to \ta'$ is called a \emph{weak equivalence} between track categories if the functors $\ta(A,B)\to \ta'(F(A),F(B))$ are equivalences of groupoids for all objects $A,B$ of $\ta$ and each object $A'$ of $\ta'$ is homotopy equivalent to some object of the form $F(A)$. Such a weak equivalence induces a functor $F: \ca\to\ca'$ between homotopy categories which is an equivalence of categories.

\

\subsection{Natural systems}  For a category $\bI$, one denotes by $\F\bI$ the
\emph{category of factorizations of} $\bI$. Let us recall that objects of the category $\F\bI$ are morphisms $f:i\to j$
of $\bI$. A morphism from $f$ to $f':i'\to j'$ in $\F\bI$ is a pair $(g,h)$, where $g:i'\to i$ and $h:j\to j'$ are morphisms in $\bI$ such that
$$f'=h\circ f\circ g.$$
In other words, the following diagram
$$\xymatrix{
j\ar[r]^h&j'\\
i\ar[u]^f&k\ar[l]^g\ar[u]_{f'} }$$
commutes. Composition in $\F\bI$ is defined by
$$(g',h')(g,h) = (gg',h'h).$$
It is clear that
$$(g ,h )=(g,{\id}_{j'})({\id}_i,h ) =({\id}_{i'},h)(g,{\id}_j).$$

\

Let $\bI$ be a category. A \emph{natural system} on $\bI$ with values in a category $\bf C$ is a functor $D:\F\bI\to {\bf C}$. We  denote the value of $D$ on $f:i\to j$ by $D_f$ or  $D(f)$. If $f$ is the identity ${\id}_i:i\to i$ we write $D_i$ instead of $D_{{\id}_i}$. We also write $g^*$ and $h_*$ instead of $D(g, {\id})$ and $D({\id},h)$. Then $$D(g, h)=g^*h_*=h_*g^*:D_f\to D_{hfg}.$$
If ${\bf C}$ is the category of sets; respectively groups, or abelian groups;  we say that $D$ is a natural system of  sets; respectively groups, or abelain groups.

\

 Let $\bI$ be a small category.  A \emph{centralised natural system of groups} is a natural system $D$ of groups on $\bI$ such that for any arrows $i\xto{f} j\xto{g} k$ and elements $x\in D_f$, $y\in D_g$ one has the equality
$$g_*(x)+f^*(y)=f^*(y)+g_*(x)$$
in the group $D_{gf}$.  By putting $i=j=k$ and $g=f=id_i$, it follows that for any object $i$ the group $D_i$ is an abelian group.

\

\subsection{Natural system $Aut^\ta$} Let $\ta$ be  a track category. Recall that the underlying category $\ka$  has the same objects as ${\ta}$, however the set of morphisms $Hom_{\ka}(A,B)$ is the set of  objects of the category ${\ta}(A,B)$.  
 For any morphism $f:A\to B$ of ${\ka}$ we let $Aut^\ta_f$ be  the collection of all automorphisms of $f$ in the category ${\ta}(A,B)$. Thus, this is the collection of all tracks $\alpha:f\To f$. It follows from TR 1  and TR 2 that $D_f$ is a group. Moreover, for any morphism $g:B\to C$, we have maps $g_*:Aut^\ta_f\to Aut^\ta_{gf}$ and $f^*:Aut^\ta_g\to Aut^\ta_{gf}$, which are group homomorphisms thanks to TR 3 -- TR 5. Moreover, in this way one obtains a natural system $Aut^\ta$ of groups on $\ka$. This follows from the identities TR 6  --  TR 8. We claim that $Aut^\ta$ is centralised. To show this it suffices to put $f=f_1$ and $g=g_1$ in TR 9 to get:
$$g_*(\alpha )+f^*(\alpha _1)=f^*(\alpha _1)+ g_{*}(\alpha).$$

\section{Cohomology of pre-track categories}
\subsection{Pre-track categories}
A pre-track category $(\pi:\ka\to\ca,G)$, or $(\pi,G)$ for short, is the following data:
\begin{enumerate}
\item $\ka$ and $\ca$ are categories and $\pi$ is a functor which is identity on objects and surjective on morphisms,
\item $G$ is a centralised natural system of groups on $\ka$.
\end{enumerate}

\begin{Ex}
The main example for this is the following. Let $\ta$ be a track category and let $\ca$ be its homotopy category. Then $(\pi:\ka\to\ca,Aut^\ta)$ is a pre-track category.
\end{Ex}

So any track category gives rise to a pre-track category. The question is whether we can say something about the converse.

\begin{De}
Let $(\pi:\ka\to\ca,G)$ be a pre-track category. A $(\pi,G)$-track category is a pair $(\ta, \sigma)$ where
\begin{enumerate}
\item $\ta=(T\ka\toto \ka)$ is a track category with underlying category $\ka$ satisfying the property: for all $f,g\in\ka(A,B)$ one has $\pi(f )=\pi(g)$ iff $\ta(f,g)\neq\emptyset$.
\item $\sigma$ is an isomorphism of natural systems on $\ka$
$$\sigma:Aut^\ta\to G.$$
\end{enumerate}
\end{De}
Two $(\pi,G)$-track categories $(\ta,\sigma)$ and $(\ta',\sigma')$ are equivalent if there exists a track functor $F:\ta\to\ta'$ which is identity on objects and the following diagram commutes:
$$\xymatrix{Aut^{\ta} \ar[rr]^{F}\ar[rd]_{\sigma} & & Aut^{\ta'}\ar[ld]^{\sigma'}\\ & G & }$$
One easily sees that if such an $F$ exists, then it is an isomorphism of $2$-categories. We let $Tracks(\pi,G)$ be the collection of equivalence classes of $(\pi,G)$-track extensions.

\subsection{Second cohomology of pre-track categories}
Given a pre-track category $(\pi:\ka\to\ca,G)$ we can also define the relative cohomology $H^2(\pi,G):=Z/_\sim$ where $Z$ is defined as follows.

\begin{De}\label{zzz} $Z$ is a collection of triples all $(\xi,\chi,\phi)$ such that
\begin{enumerate}
\item $\xi$ is a function which assigns to all triples of morphisms
$f,g,h:i\to j$ of $\ka$ with $\pi(f)=\pi(g)=\pi(h)$, an element $\xi(f,g,h)\in G_f$,
\item $\chi$ is a function which assigns to each diagram $\xymatrix{i\ar@<1ex>[r]^f\ar@<-1ex>[r]_g&\ j\ar@<1ex>[r]^x\ar@<-1ex>[r]_y&\ k}$ with $\pi(f)=\pi(g)$ and $\pi(x)=\pi(y)$ an element $\chi(f,g|x,y)\in G_{xf}$,
\item $\phi$ is a function which assigns to each pair of morphisms $f,g$ in $\ka$ with $\pi(f)=\pi(g)$ an isomorphism $\phi_{g,f}:G_g\to G_f$.
\end{enumerate}
\renewcommand{\theenumi}{\roman{enumi}}
These functions must satisfy the following equations:
\begin{enumerate}
\item To simplify notation, first we define $m_*\phi_{b,a}:G_{mb}\to G_{ma}$ by
$$(m_*\phi_{b,a})(t):=-\chi(a,b|m,m)+\phi_{mb, ma}(t)+\chi(a,b|m,m)$$
where $a,b:i\to j$ and $m:j\to k$. Also, we define $b^*\phi_{n,m}:G_{nb}\to G_{mb}$ by
$$(b^*\phi_{n,m})(t):=-\chi(b,b|m,n)+\phi_{nb,mb}(t)+\chi(b,b|m,n),$$
where $b:i\to j$ and $m,n:j\to k$. Then we have the following equations:
\begin{enumerate}\label{phi}
\item For $f,g,h:i\to j$, $\phi_{g,f}\circ\phi_{h,g}(t)=\xi(f,g,h)+\phi_{h,f}(t)+\xi(f,g,h)$.
\item For $a,b:i\to j$, $m,n:j\to k$, $\beta\in G_b$ and $\nu\in G_n$, $m_*\phi_{b,a}(m_*\beta)=m_*(\phi_{b,a}(\beta)),  \  \ a^*\phi_{n,m}(a^*\nu)=a^*(\phi_{n,m}(\nu))$,
\item and for $\mu\in G_m$ and $\alpha\in G_a$, $m_*\phi_{b,a}(b^*\mu)=a^*\mu, \  \ a^*\phi_{n,m}(n_*\alpha)=m_*\alpha$.
\end{enumerate}
\item \label{xi} For $f,g,h,e:i\to j$ with $\pi(f)=\pi(g)=\pi(h)=\pi(e)$, we have to have $$\xi(f,g,e)+\phi_{g,f}\xi(g,h,e)=\xi(f,h,e)+\xi(f,g,h).$$
\item \label{xichi} For $x,y,z:i\to j$ with $\pi(x)=\pi(y)=\pi(z)$ and $a,b,c:j\to k$ with $\pi(a)=\pi(b)=\pi(c)$, we have to have $$\xi(ax,by,cz)+\phi_{by,ax}(\chi(y,z|b,c))+\chi(x,y|a,b)=\chi(x,z|a,c)+x^*\xi(a,b,c)+a_*\xi(x,y,z).$$
\item \label{chi} For $x,y:i\to j$ with $\pi(x)=\pi(y)$, $a,b:j\to k$ with $\pi(a)=\pi(b)$ and $m,n:k\to l$ with $\pi(m)=\pi(n)$, we have to have
$$\chi(ax,by|m,n)+m_*\chi(x,y|a,b)=\chi(x,y|ma,nb)+x^*(\chi(a,b|m,n)).$$
\end{enumerate}
\end{De}
\begin{De}
We say two such triples $(\xi,\chi, \phi)$ and $(\xi ',\chi ',\phi ')$ are equivalent, if there is an element $\zeta(f,g)\in G_f$ such that
\begin{enumerate}\label{equi}
\item $\zeta(f,h)+\xi(f,g,h)-\phi_{g,f}\zeta(g,h)-\zeta(f,g)=\xi'(f,g,h)$,
\item $\zeta(ax,by)+\chi(x,y|a,b)-x^*\zeta(a,b)-a_*\zeta(x,y)=\chi'(x,y|a,b)$, and
\item $\zeta(f,g)+\phi_{g,f}(t)-\zeta(f,g)=\phi'_{g,f}(t).$
\end{enumerate}
\end{De}

\begin{Th}
There is a natural bijection
$$Tracks(\pi,G)\simeq H^2(\pi,G).$$
\end{Th}

%

\begin{proof}
Let $(\ta,\sigma)$ be an element of $Tracks(\pi,G)$. Let $\pi(f)=\pi(g)$, so $\ta(f,g)\neq\emptyset$. Choose a track $H_{f,g}\in T(f,g)$ so that $$H_{f,f}=0, \  \  \  H_{f,g}=-H_{g,f}.$$
Now take $(f,g,h)$ in such a way that $\pi(f)=\pi(g)=\pi(h)$. Then there is a unique element $\xi(f,g,h)\in G_f$ such that
$$H_{f,g}+H_{g,h}=-\xi(f,g,h)+H_{f,h}.$$
Take $x,y\in K(i,j)$ and $a,b\in K(j,k)$ for which $\pi(x)=\pi(y)$ and $\pi(a)=\pi(b)$. Then there is a unique element $\chi(x,y|a,b)\in G_{ax}$ such that
$$a_*H_{x,y}+y^*H_{a,b}=-\chi(x,y|a,b)+H_{ax,by}.$$

For every $f$ and $g$ with $\pi(f)=\pi(g)$ we define an isomorphism $\phi_{g,f}(t):=H_{f,g}+t-H_{f,g}$, where $t\in T(g,g)$.

We claim that the triple $(\xi, \chi, \phi)$ is an element of $Z$.

The Equations (\ref{phi}) of Definition \ref{zzz} are straightforward to check. First, we check composition:

\begin{align*}
\phi_{g,f}\circ\phi_{h,g}(t)&=\phi_{g,f}(H_{g,h}+t-H_{g,h})=H_{f,g}+H_{g,h}+t-H_{g,h}-H_{f,g}\\
&=-\xi(f,g,h)+\phi_{h,f}(t)+\xi(f,g,h).
\end{align*}

Next, we have
\begin{align*}
m_*(\phi_{b,a}(\beta))&=m_*(H_{a,b}+\beta-H_{a,b})=m_*H_{a,b}+m_*\beta-m_*H_{a,b}\\
&=(m_*\phi_{b,a})(m_*\beta).
\end{align*}
Similarly,
\begin{align*}
a^*(\phi_{n,m}(\nu))&=a^*(H_{m,n}+\nu-H_{m,n})=(a^*\phi_{n,m})(a^*\nu).
\end{align*}

The next equation
$$m_*\phi_{b,a}(b^*\mu)=m_*H_{a,b}+b^*\mu-m_*H_{a,b}=a^*\mu$$
follows from $TR9$, as shown by the diagram
$$\xymatrix{i\ar@/^/[r]^{a} \ar@/_/[r]_{b}&j\ar@/^/[r]^{m}\ar@/_/[r]_{m} & k}, \ \ H_{a,b}:a\To b, \ \ \mu:m \To m.$$
The final equation is very similar to this.

To check (\ref{xi}), we use the associativity property for the expression $H_{f,g}+H_{g,h}+H_{h,e}$ and the definition of the function $\xi$ to obtain
$$(H_{f,g}+H_{g,h})+H_{h,e}=-\xi(f,g,h)+H_{f,h}+H_{h,e}=-\xi(f,g,h)-\xi(f,h,e)+H_{f,e}$$
$$H_{f,g}+(H_{g,h}+H_{h,e})=H_{f,g}-\xi(g,h,e)+H_{g,e}=H_{f,g}-\xi(g,h,e)-H_{f,g}-\xi(f,g,e)+H_{f,e}.$$
So we have
$$\xi(f,h,e)+\xi(f,g,h)=\xi(f,g,e)+\phi_{g,f}(\xi(g,h,e)).$$

Let $x,y,z:i\to j$ and $a,b,c:j\to k$ be morphisms in $K$ such that $\pi(x)=\pi(y)=\pi(z)$ and $\pi(a)=\pi(b)=\pi(c)$. Then we have
\begin{align*}
a_*(H_{x,y}+H_{y,z})&+ z^*(H_{a,b}+H_{b,c})\\
=a_*H_{x,y}+(a_*H_{y,z}&+z^*H_{a,b})+z^*H_{b,c}\\
=a_*H_{x,y}+(y^*H_{a,b}&+b_*H_{y,z})+z^*H_{b,c}\\
=-\chi(x,y|a,b)+H_{ax,by}&-\chi(y,z|b,c)+H_{by,cz}\\
=-\chi(x,y|a,b)+H_{ax,by}&-\chi(y,z|b,c)-H_{ax,by}+H_{ax,by}+H_{by,cz}\\
=-\chi(x,y|a,b)+H_{ax,by}&-\chi(y,z|b,c)-H_{ax,by}-\xi(ax,by,cz)+H_{ax,cz}.\\
\end{align*} 

On the other hand,
\begin{align*}
&a_*(H_{x,y}+H_{y,z})+ z^*(H_{a,b}+H_{b,c})\\
=&a_*(-\xi(x,y,z)+H_{x,z})+ z^*(-\xi(a,b,c)+H_{a,c})\\
=&-a_*\xi(x,y,z)+a_*H_{x,z}- z^*\xi(a,b,c)-a_*H_{x,z}+a_*H_{x,z}+z^*H_{a,c}\\
=&-a_*\xi(x,y,z)+a_*H_{x,z}- z^*\xi(a,b,c)-a_*H_{x,z}-\chi(x,z|a,c)+H_{ax,cz}\\
=&-a_*\xi(x,y,z)-a_*\phi_{z,x}( z^*\xi(a,b,c))-\chi(x,z|a,c)+H_{ax,cz}\\
=&-a_*\xi(x,y,z)- x^*\xi(a,b,c)-\chi(x,z|a,c)+H_{ax,cz} \ \ {\rm using} \ {\rm Equation} \ {\rm (\ref{phi})}.
\end{align*}

We therefore obtain Equation (\ref{xichi}).
$$\xi(ax,by,cz)+\phi_{by,ax}(\chi(y,z|b,c))+\chi(x,y|a,b)=\chi(x,z|a,c)+x^*\xi(a,b,c)+a_*\xi(x,y,z).$$

Let $x,y:i\to j$, $a,b:j\to k$ and $m,n:k\to l$ be morphisms in $K$ such that $q(x)=q(y)$, $q(a)=q(b)$ and $q(m)=q(n)$. By the definition of $\chi$ we have
$$m_*H_{a,b}+b^*H_{m,n}=-\chi(a,b|m,n)+H_{ma,nb}$$
and
$$m_*a_*H_{x,y}+y^*H_{ma,nb}=-\chi(x,y|ma,nb)+H_{max,nby}.$$

Expressing $H_{ma,nb}$ from the first equation and substituting in the second, we obtain that the expression $-\chi(x,y|ma,nb)+H_{max,nby}$ is equal to
\begin{align*}
&m_*a_*H_{x,y}+y^*(\chi(a,b|m,n))+y^*m_*H_{a,b}+y^*b^*H_{m,n}\\
=&m_*a_*H_{x,y}+y^*(\chi(a,b|m,n))-m_*a_*H_{x,y}+m_*a_*H_{x,y}+y^*m_*H_{a,b}+y^*b^*H_{m,n}\\
=&m_*a_*H_{x,y}+y^*(\chi(a,b|m,n))-m_*a_*H_{x,y}+m_*(a_*H_{x,y}+y^*H_{a,b})+y^*b^*H_{m,n}\\
=&m_*a_*H_{x,y}+y^*(\chi(a,b|m,n))-m_*a_*H_{x,y}-m_*\chi(x,y|a,b)+m_*H_{ax,by}+y^*b^*H_{m,n}\\
=&m_*a_*H_{x,y}+y^*(\chi(a,b|m,n))-m_*a_*H_{x,y}-m_*\chi(x,y|a,b)-\chi(ax,by|m,n)+H_{max,nby}\\
=&x^*(\chi(a,b|m,n))-m_*\chi(x,y|a,b)-\chi(ax,by|m,n)+H_{max,nby},\ {\rm using} \ {\rm Equation} \  {\rm (\ref{phi})}.
\end{align*}
Therefore we have
$$\chi(ax,by|m,n)+m^*\chi(x,y|a,b)=\chi(x,y|ma,nb)+x^*(\chi(a,b|m,n)).$$

If we choose another $H_{f,g}'\in \ta(f,g)$, there is a unique element $\zeta(f,g)\in G_f$ such that
$$\zeta(f,g)+H_{f,g}=H_{f,g}'.$$
We let $\xi', \chi'$ and $\phi'$ be the functions corresponding to $H'$. For all $(f,g,h)$ with $\pi(f)=\pi(g)=\pi(h)$ we have
$$H_{f,g}'+H_{g,h}'=-\xi'(f,g,h)+H_{f,h}'.$$
Hence,
$$\zeta(f,g)+H_{f,g}+\zeta(g,h)+H_{g,h}=\zeta(f,g)+H_{f,g}+\zeta(g,h)-H_{f,g}+H_{f,g}+H_{g,h}$$
$$=-\xi'(f,g,h)+\zeta(f,h)+H_{f,h}.$$
This gives
$$\zeta(f,h)+\xi(f,g,h)-\phi_{g,f}\zeta(g,h)-\zeta(f,g)=\xi'(f,g,h).$$
For $(x,y)\in\ka(i,j)$ and $(a,b)\in\ka(j,k)$ with $\pi(x)=\pi(y)$ and $\pi(a)=\pi(b)$, we have
$$a_*H'_{x,y}+y^*H'_{a,b}=-\chi'(x,y|a,b)+H'_{ax,by}.$$
So we get
$$a_*\zeta(x,y)+a_*H_{x,y}+y^*\zeta(a,b)+y^*H_{a,b}=a_*\zeta(x,y)+a_*\phi_{y,x}(y^*\zeta(a,b))-\chi(x,y|a,b)+H_{ax,by}$$
$$=-\chi'(x,y|a,b)+\zeta(ax,by)+H_{ax,by}.$$
Hence
$$\zeta(ax,by)+\chi(x,y|a,b)-x^*\zeta(a,b)-a_*\zeta(x,y)=\chi'(x,y|a,b).$$
We also have for all $(f,g)\in\ka$ with $\pi(f)=\pi(g)$ and all $t\in G_f$
$$\zeta(f,g)+\phi_{g,f}(t)-\zeta(f,g)=\phi'_{g,f}(t).$$

\

The inverse map is constructed in the following way. Let $(\xi,\chi,\phi)\in H^2(\pi,G)$.

Let $\pi(f)=\pi(g)$. Then we define the set of tracks $T(f,g)$ to be the set $G_f\times{f,g}$. The track category structure is given by
$$(\alpha, f,g)+(\beta, g,h)=(\alpha+\phi_{g,f}(\beta)-\xi(f,g,h), f,h).$$
$$a_*(\alpha,f,g)=(a_*\alpha-\chi(f,g|a,a),af,ag),$$
$$b^*(\alpha,f,g)=(b^*\alpha-\chi(b,b|f,g),fb,gb).$$
Let us check that this indeed gives a track structure.

Now, to verify the relations $TR1-TR9$: First, we check that the addition defined above is associative. We have
\begin{align*}
((\alpha, f,g)+(\beta, g,h))+(\gamma,h,e)&=(\alpha+\phi_{g,f}(\beta)-\xi(f,g,h), f,h)+(\gamma,h,e)\\
&=(\alpha+\phi_{g,f}(\beta)-\xi(f,g,h)+\phi_{h,f}(\gamma)-\xi(f,h,e),f,e)\\
&=(\alpha+\phi_{g,f}(\beta)+\phi_{g,f}\circ\phi_{h,g}(\gamma)-\xi(f,g,h)-\xi(f,h,e),f,e).\\
\end{align*}
On the other hand,
\begin{align*}
(\alpha, f,g)+((\beta, g,h)+(\gamma,h,e))&=(\alpha, f,g)+(\beta)+\phi_{h,g}(\gamma)-\xi(g,h,e),g,e)\\
&=(\alpha+\phi_{g,f}(\beta)+\phi_{g,f}\circ\phi_{h,g}(\gamma)-\phi_{g,f}(\xi(g,h,e))-\xi(f,g,e),f,e).\\
\end{align*}
Using Equation \ref{zzz} (\ref{xi}), we can see that we have equality.

\
Next, for $TR2$ we have:
$$(\alpha, f,g)+(0,g,g)=(\alpha+\phi_{g,f}(0)-\xi(f,g,g),f,g)=(\alpha,f,g)=(0,f,f)+(\alpha,f,g)$$.
We also have that
\begin{align*}
m^*((\alpha,f,g)+(\beta,g,h))&=m^*(\alpha+\phi_{g,f}(\beta)-\xi(f,g,h), f,h)\\
=(m^*(\alpha+\phi_{g,f}(\beta)&-\xi(f,g,h))-\chi(m,m|f,h),fm,hm)\\
=(m^*\alpha+m^*\phi_{g,f}(m^*\beta)&-m^*\xi(f,g,h)-\chi(m,m|f,h),fm,hm)\\
=(m^*\alpha-\chi(m,m|f,g)+\phi_{gm,fm}(m^*\beta)&+\chi(m,m|f,g)-m^*\xi(f,g,h)-\chi(m,m|f,h), fm,hm).\\
\end{align*}
According to Equation \ref{zzz} (\ref{xichi}), $-m^*\xi(f,g,h)-\chi(m,m|f,h)=-\chi(m,m|f,g)-\phi_{gm,fm}(\chi(m,m|g,h))-\xi(fm,gm,hm)$. Therefore, we have
\begin{align*}
=(m^*\alpha-\chi(m,m|f,g)+\phi_{gm,fm}(m^*\beta)&-\phi_{gm,fm}(\chi(m,m|g,h))-\xi(fm,gm,hm), fm,hm)\\
=m^*(\alpha,f,g)&+m^*(\beta,g,h).
\end{align*}

For $TR4$, we have very similarly:
\begin{align*}
x_*((\alpha,f,g)+(\beta,g,h))&=x_*(\alpha+\phi_{g,f}(\beta)-\xi(f,g,h), f,h)\\
=(x_*(\alpha+\phi_{g,f}(\beta)&-\xi(f,g,h))-\chi(f,h|x,x),xf,xh)\\
=(x_*\alpha+x_*\phi_{g,f}(x_*\beta)&-x_*\xi(f,g,h)-\chi(f,h|x,x),xf,xh)\\
=(x_*\alpha-\chi(f,g|x,x)+\phi_{xg,xf}(x_*\beta)&+\chi(f,g|x,x)-x_*\xi(f,g,h)-\chi(f,h|x,x),xf,xh).
\end{align*}
According to Equation \ref{zzz} (\ref{xichi}), $-x_*\xi(f,g,h)-\chi(f,h|x,x)=-\chi(f,g|x,x)-\phi_{xg,xf}(\chi(g,h|x,x))-\xi(xf,xg,xh)$. Therefore, we have
\begin{align*}
=(x_*\alpha-\chi(f,g|x,x)+\phi_{xg,xf}(x_*\beta)&-\phi_{xg,xf}(\chi(g,h|x,x))-\xi(xf,xg,xh),xf,xh)\\
=x_*(\alpha,f,g)&+x_*(\beta,g,h).
\end{align*}

For $TR5$, we have
$$f^*(0,g,g)=(f^*0-\chi(f,f|g,g),gf,gf)=(0,gf,gf)=(g_*0-\chi(f,f|g,g),gf,gf)=g_*(0,f,f).$$

Next,
$$m^*n^*(\alpha,f,g)=m^*(n^*\alpha-\chi(n,n|f,g),fn,gn)=(m^*n^*\alpha-m^*\chi(n,n|f,g)-\chi(m,m|fn,gn),fnm,gnm).$$
Using Equation \ref{zzz} (\ref{chi}), we get that $-m^*\chi(n,n|f,g)-\chi(m,m|fn,gn)=-\chi(nm,nm|f,g)$, therefore giving us
$$=(n^*m^*\alpha-\chi(nm,nm|f,g),fnm,gnm)=(nm)^*(\alpha,f,g).$$
$TR7$ can be shown in a very similar manner. For $TR8$, we have
\begin{align*}
x_*m^*(\alpha,f,g)&=x_*(m^*\alpha-\chi(m,m|f,g),fm,gm)\\
&=(x_*m^*\alpha-x_*\chi(m,m|f,g)-\chi(fm,gm|x,x),xfm,xgm)\\
&=(x_*m^*\alpha-m^*\chi(f,g|x,x)-\chi(m,m|xf,xg),xfm,xgm),\\
&=m^*x_*(\alpha,f,g),
\end{align*}
due to Equation \ref{zzz} (\ref{chi}).

We also check $TR9$, the last equation:
\begin{align*}
m_*(\alpha,a,b)+b^*(\mu,m,n)=&(m_*\alpha-\chi(a,b|m,m),ma,mb)+(b^*\mu-\chi(b,b|m,n),mb,nb)\\
=&(m_*\alpha-\chi(a,b|m,m)+\phi_{mb,ma}(b^*\mu-\chi(b,b|m,n))-\xi(ma,mb,nb),ma,nb)\\
=&(m_*\alpha-\chi(a,b|m,m)+\chi(a,b|m,m)+m_*\phi_{b,a}(b^*\mu)-\chi(a,b|m,m)\\
&-\phi_{mb,ma}(\chi(b,b|m,n))-\xi(ma,mb,nb),ma,nb).
\end{align*}
According to Equation \ref{zzz} (\ref{xichi}), this equals to
\begin{align*}
&=(m_*\alpha+m_*\phi_{b,a}(b^*\mu)-m_*\xi(a,b,b)-m_*\phi_{b,a}(b^*\xi(m,m,n))-\chi(a,b|m,n),ma,nb)\\
&=(m_*\alpha+m_*\phi_{b,a}(b^*\mu)-\chi(a,b|m,n),ma,nb)=(m_*\alpha+a^*\mu-\chi(a,b|m,n),ma,nb).
\end{align*}
On the other hand,
\begin{align*}
a^*(\mu,m,n)+n_*(\alpha,a,b)=&(a^*\mu-\chi(a,a|m,n),ma,na)+(n_*\alpha-\chi(a,b|n,n),na,nb)\\
=&(a^*\mu-\chi(a,a|m,n)+\phi_{na,ma}(n_*\alpha-\chi(a,b|n,n))-\xi(ma,na,nb),ma,nb)\\
=&(a^*\mu+a^*\phi_{n,m}(n_*\alpha)-\chi(a,a|m,n)-\phi_{na,ma}(\chi(a,b|n,n))-\xi(ma,na,nb),ma,nb).
\end{align*}
According to Equation \ref{zzz} (\ref{xichi}), this equals to
\begin{align*}
&=(a^*\mu+a^*\phi_{n,m}(n_*\alpha)-m_*\xi(a,a,z)-m_*\phi_{b,a}(b^*\xi(m,n,n))-\chi(a,b|m,n),ma,nb)\\
&=(a^*\mu+a^*\phi_{n,m}(n_*\alpha)-\chi(a,b|m,n),ma,nb)=(a^*\mu+m_*\alpha-\chi(a,b|m,n),ma,nb).
\end{align*}
As $G$ is a centralised natural system, we have $a^*\mu+m_*\alpha=m_*\alpha+a^*\mu$ in $G_{ma}$, therefore giving equality.
\end{proof}

\end{document}